\documentclass[11pt]{amsart}

\usepackage{enumerate, amssymb,  mathrsfs, nicefrac, upref, graphicx, pdfsync}
\newtheorem{theorem}{Theorem}[section]
\newtheorem{lemma}[theorem]{Lemma}
\newtheorem*{lemma*}{Lemma}

\newtheorem{proposition}[theorem]{Proposition}
\newtheorem{corollary}[theorem]{Corollary}

\theoremstyle{definition}
\newtheorem{definition}[theorem]{Definition}
\newtheorem{example}[theorem]{Example}

\theoremstyle{remark}

\numberwithin{equation}{section}

\newcommand{\be}[1]{\begin{equation}\label{#1}}
\newcommand{\ee}{\end{equation}}

\newcommand{\abs}[1]{\lvert#1\rvert}

\newcommand{\C}{\mathbb{C}}

\newcommand{\W}{\mathscr{W}}

\newcommand{\R}{\mathbb{R}}

\newcommand{\dtext}{\textnormal d}

\newcommand{\onto}{\xrightarrow[]{{}_{\!\!\textnormal{onto\,\,}\!\!}}}

\newcommand{\deff}{\stackrel {\textnormal{\tiny{def}}}{=\!\!=} }

\DeclareMathOperator{\re}{Re}
\DeclareMathOperator{\im}{Im}

\DeclareMathOperator{\loc}{loc}

\DeclareMathOperator{\Card}{Card}

\def\XXint#1#2#3{{\setbox0=\hbox{$#1{#2#3}{\int}$}\vcenter{\hbox{$#2#3$}}\kern-.5\wd0}}

\def\XXiint#1#2#3{{\setbox0=\hbox{$#1{#2#3}{\iint}$}\vcenter{\hbox{$#2#3$}}\kern-.5\wd0}}

\def\le{\leqslant}
\def\ge{\geqslant}

\begin{document}

\title{Lipschitz~regularity~for inner-variational~equations}

\author[T. Iwaniec]{Tadeusz Iwaniec}
\address{Department of Mathematics, Syracuse University, Syracuse,
NY 13244, USA and Department of Mathematics and Statistics,
University of Helsinki, Finland}
\email{tiwaniec@syr.edu}

\author[L. V. Kovalev]{Leonid V. Kovalev}
\address{Department of Mathematics, Syracuse University, Syracuse,
NY 13244, USA}
\email{lvkovale@syr.edu}

\author[J. Onninen]{Jani Onninen}
\address{Department of Mathematics, Syracuse University, Syracuse,
NY 13244, USA}
\email{jkonnine@syr.edu}
\thanks{Iwaniec was supported by the NSF grant DMS-0800416 and the Academy of Finland project 1128331. Kovalev was supported by the NSF grant DMS-0968756. Onninen was supported by the NSF grant DMS-1001620.}

%    General info
\subjclass[2010]{Primary 49N60; Secondary 35B65, 73C50}

%\date{\today}

\keywords{Inner variation,  Lipschitz regularity,  extremal problems, quasiregular mappings, topological degree}

\begin{abstract}
We obtain Lipschitz regularity results for a fairly general class of nonlinear first-order PDEs. These equations arise from the inner variation of certain energy integrals. Even in the simplest model case of the Dirichlet energy the \textit{inner-stationary solutions} need not be differentiable everywhere; the Lipschitz continuity is the best possible. But the proofs, even in the Dirichlet case, turn out to relay on  topological arguments.
The appeal to the inner-stationary solutions in this context is motivated by the classical problems of existence and regularity of the energy-minimal deformations in the theory of harmonic mappings and certain mathematical models of nonlinear elasticity; specifically, neo-Hookian type problems.
\end{abstract}
\maketitle

 %\tableofcontents

\section{Introduction}

We establish Lipschitz regularity of solutions of nonlinear first-order PDEs that arise from inner variation of numerous energy integrals. This includes the model case of the Dirichlet energy for mappings $h \colon \Omega \to \Omega^\ast$ between two designated domains in $\C$. Roughly speaking, the inner variation of $h$ amounts to composing $h$ with a diffeomorphism of $\Omega$ onto itself. This type of variation is often used  when the standard first variation is not allowed.  For instance, when dealing with mappings with nonnegative Jacobian  the inner variation is necessary to preserve the sign of the Jacobian.  Let us begin with the Dirichlet integral,
\begin{equation}\label{direnergy}
\mathscr E _{_{\Omega}} [h]= \iint_{\Omega} \abs{Dh}^2 =2 \iint_{\Omega} \left( \abs{h_z}^2+ \abs{h_{\bar z}}^2 \right)\, \dtext x_1\, \dtext x_2\;,\;\;\;z = x_1 + i\,x_2
\end{equation}
Hereafter  $\,h_z = \frac{\partial h}{\partial z} \,$ and $\, h_{\bar z} = \frac{\partial h}{\partial \overline{z}}\,$ are complex partial derivatives of $\,h\,$. The first variation of $\mathscr E$ results in the Euler-Lagrange equation,
\begin{equation}\label{ELe}
\Delta h = 4 h_{z \bar z} =0
\end{equation}
In contrast, the inner variation leads to a nonlinear equation
\begin{equation}\label{HLe}
\frac{\partial}{ \partial \bar z} \left(h_z \, \overline{h_{\bar z}}\right)=0, \quad
\text{equivalently, }\quad h_z \, \overline{h_{\bar z}} = \phi, \quad  \text{($\phi$ is analytic)}
\end{equation}
for mapping in the Sobolev space $\,\mathscr W^{1,2}_{\textnormal{loc}}(\Omega)\,$.
This equation will henceforth be referred to as the \textit{Hopf-Laplace equation}. There are important nonharmonic solutions of (\ref{HLe}). Such solutions arise typically as weak limits of the energy-minimizing sequences of diffeomorphisms $h \colon \Omega \onto \Omega^\ast$.  Passing to the limit we often loose  harmonicity; at the points where the limit map fails to be injective~\cite{IKKO}, and only at those points~\cite{IKOhopf}.

The unavailability of the Euler-Lagrange equation is a major source of difficulties in the theory of nonlinear elasticity~\cite{Ba, Ba1, SS}.
This drives one to investigate the regularity of the energy-minimal mappings  on the basis of the inner-variational  equation alone; also known as  \textit{energy-momentum} or \textit{equilibrium} equations, etc~\cite{Cob, SSe, Ta}.
Several results in this direction were obtained in~\cite{BOP, CIKO, Mo, Ya}. Nevertheless this theory is still in its infancy.

 The most general setting we are dealing with can be described as follows.
Let $\mathcal H = \mathcal H(z,\xi)$ be a continuous function  in $\,\Omega \times \widehat{\mathbb C}_R\,$,  where $\,\widehat{\mathbb C}_R = \{ \xi \colon R < |\xi| \leqslant \infty\,\}$, $0 \leqslant R < \infty$.
We impose two structural conditions on $\mathcal H$. The first one is the Lipschitz condition with respect to the reciprocal of the $\,\xi\,$ variable,
\begin{equation}\label{LipCond}
\;\;|\mathcal H(z, \xi_1) - \mathcal H(z, \xi_2) | \leqslant  L \cdot \Big|\frac{1}{\xi_1} - \frac{1}{\xi_2}\Big|\;,\;\;\text{for $\,z \in \Omega\,$ and $\,\xi_1 ,\,\xi_2 \in \widehat{\mathbb C}_R$.}
\end{equation}
with  a constant $\,0\leqslant L < \infty$.
 Regarding regularity with respect to $z \in \Omega$,  we shall require  that the function $z \mapsto  \mathcal H(z,\xi)$  be H\"older continuous. Precisely, the second structural condition reads as:
 \begin{equation}\label{HoldCond}
  \underset{z\in\Omega}{\sup}\, |\mathcal H(z, \xi)| \;\;+\;\;\underset{z_1\neq z_2}{\sup} \frac{|\mathcal H(z_1, \xi) - \mathcal H(z_2, \xi)|}{|z_1 - z_2|^\alpha} \; \leqslant M\,,\;\;\;\text{for}\;\;|\xi| >R
 \end{equation}

\begin{definition}
A function $h\in \W^{1,2}_{\loc} (\Omega)$ is said to be a solution of the equation
\begin{equation}\label{TheEquat}
h_{\bar z} =  \mathcal H(z, h_z)
\end{equation}
if~\eqref{TheEquat} holds for almost every point $\,z\in \Omega\,$,  whenever  $\abs{h_z(z) }>R$.
\end{definition}

Note we impose no condition  at the points where $\,\abs{h_z(z)} \leqslant R\,$; this yields boundedness of the gradient of $\,h\, $,  $\,|h_{\bar z } | \leqslant |h_z| \leqslant R\,$.

A simplified version of our main result reads as follows

\begin{theorem}\label{MainTheorem} Let the equation~\eqref{TheEquat}  comply with the conditions~\eqref{LipCond}  and~\eqref{HoldCond}. Then every solution $h \in \mathscr W^{1,2}(\Omega)$
with nonnegative Jacobian  is locally Lipschitz continuous but not necessarily $\mathscr C^1$-smooth.
 \end{theorem}

A special case of Theorem~\ref{MainTheorem} deserves a separate mention because it covers the variant of~\eqref{HLe} with not necessarily analytic right-hand side $\phi$.

 \begin{theorem}\label{LipHopf2}
Let $h\in \mathscr W^{1,2}(\Omega)$ be  a mapping with nonnegative Jacobian.
Suppose that the Hopf product $h_z\,\overline{h_{\bar z}}$ is bounded and H\"{o}lder continuous.
Then $h$ is locally Lipschitz but not necessarily $\mathscr C^1$-smooth.
\end{theorem}

More specific statements, including gradient estimates near $\,\partial\Omega\,$,  are presented in Theorem \ref{LipHopf3}. Examples of variational problems leading to equations in Theorem~\ref{MainTheorem} are provided in Section~\ref{invarsec}.

Our proofs draw upon the theory of Beltrami equations combined with methods of topology, a technique originated in the theory of general nonlinear first order elliptic systems ~\cite{AIMb, AIS, BI, Iw}. Initially, for the linear elliptic system
\[
h_{\bar z} = \mu \, h_z+\nu \, \overline{h_{\bar z}},\qquad \abs{\mu}+\abs{\nu}\le k<1
\]
  the solutions are quasiregular, a concept firmly rooted in the geometry  of analytic functions. While it might sound trivial, the observation that the difference of two solutions is again quasiregular is deep and useful, for it provides  us with powerful topological tools to obtain  existence, uniqueness and regularity of solutions. And this was exactly a blueprint for the definition of ellipticity of fully nonlinear Beltrami type equations,
\be{NonBeltr}
h_{\bar z} = \mathcal H(z, h_z)\;,\;\;\; |\mathcal H(z, \xi_1) - \mathcal H(z, \xi_2) | \leqslant  k |\xi_1 - \xi_2| \;, \;\;\;0 \leqslant k <1
\ee
Because of nonlinearity the difference of two solutions need not solve the same equation but it does solve another elliptic equation and, as such,  satisfies the  distortion inequality
\begin{equation}
| f_{\bar z} - g_{\bar z}\,| \;\leqslant \,k\,|f_z -  g_z |
\end{equation}
meaning that $f-g$ is quasiregular, see Definition~\ref{qrdef}.

A chief distinction from the elliptic cases discussed above is that the  solutions of~\eqref{TheEquat} need not be  quasiregular. However, we will construct a continuous family $\{F^\lambda \}_{\lambda \in \C}$ of ``good'' solutions of~\eqref{TheEquat} such that $F^\lambda -h$ are quasiregular.  After that we appeal to the topological properties of quasiregular mappings.  

We believe that these ideas will have applications to even more general PDEs than
those in our paper.  An interested reader  is referred to recent papers by D.~Faraco, B.~Kirchheim and L.~Sz\'ekelyhidi~\cite{FS, KS} which also combine the theory of quasiregular mappings with topological arguments.

\section{Inner-variational equations}\label{invarsec}

Let us consider the energy integral for mappings $\, h : \Omega \rightarrow \mathbb C\,$
\begin{equation}\label{equ21}
\mathscr E[h] = \iint _\Omega \mathbf E(z, h, h_z, h_{\bar z}) \,\dtext x_1\, \dtext x_2\;,\;\;\;\;\;\;\; z = x_1 + i x_2
\end{equation}
where $\,\mathbf E = \mathbf E(z, w, \xi, \zeta)\,$ is a given \textit{stored-energy function}.
From the point of view of Geometric Function Theory the mappings $\, h\,$ must take $\Omega\,$  onto a designated domain $\,\Omega^*\,$. In the elasticity theory these domains  are  referred to as the \textit{reference configuration} and \textit{deformed configuration}, respectively.  In the neo-Hookean model of hyperelasticity the stored-energy function blows up when the Jacobian of $\,h\,$ approaches zero. Thus one is looking for mapping $\,h\,$ with positive Jacobian determinant  which minimizes the energy. It is not always the case that  the minimizers satisfy the Euler-Lagrange equation. However, it is legitimate to  perform the inner variation of the energy integral.

Given any test function $\,\eta \in \mathscr C^\infty_0(\Omega)\,$ and a complex parameter $\,t\,$, small enough so that the map $ z\mapsto z + t\,\eta(z)\,$ represents a diffeomorphism of $\,\Omega\,$ onto itself,  consider the inner variation $\, h^t(z) = h( z + t \eta)\,$ and its energy
\begin{equation}\nonumber
\mathscr E[h^t] = \iint _\Omega \mathbf E(z, h^t, h^t_z, h^t_{\bar z}) \,\dtext x_1\, \dtext x_2
\end{equation}
First we make a substitution $\,w = z\,+\,t \,\eta(z)\,$ and then differentiate to obtain an integral form of the equilibrium equation $\,\frac{\partial}{\partial t }\big |_{t=0} \mathscr E[h^t]\, = 0 \,$ . We eliminate $\,\eta\,$ through integration by parts to arrive at  what is called the \textit{inner-variational equation}
\begin{equation}\label{VarEq}
\frac{\partial}{\partial \bar z} \Big [  h_z \,\mathbf E_\zeta \; + \;\overline{h_{\bar z}} \,\mathbf E_{\bar \xi}  \Big] \;+\; \frac{\partial}{\partial z} \Big [  h_z \,\mathbf E_\xi \; + \;\overline{h_{\bar z}} \,\mathbf E_{\bar \zeta}  \; -\; \mathbf E \, \Big] \;+\; \mathbf E_z\; = 0
\end{equation}
Here the subscripts under $\,\mathbf E \,$ stand for complex partial derivatives of  $\,\mathbf E = \mathbf E(z, w, \xi, \zeta)\,$. The partial derivatives $\,\frac{\partial}{\partial z}\,$ and $\, \frac{\partial}{\partial \bar z}\,$ are understood in the sense of distributions. Note that $\,\mathbf E_w\,$ does not enter this equation; meaning that  $\, \mathbf E\,$ is not required to be differentiable with respect to the $w$-variable. We say that $\,h\,$ is \textit{\textbf{inner stationary}} for $\,\mathscr E\,$ if it satisfies the equation~\eqref{VarEq}, regardless of whether $\,h\,$ is energy-extremal or not.

\subsection{Dirichlet integral}
The most basic example is the Dirichlet integrand $\mathbf E = |\xi|^2 + |\zeta|^2$ and the associated Hopf-Laplace
equation~\eqref{HLe},
\[
\frac{\partial}{\partial \bar z}\,\big( h_z\, \overline{h_{\bar z}}\,\big)  = 0\;,\;\;\;\text{for}\;\; h\in \mathscr W^{1,2}_{\textnormal{loc}}(\Omega)
\]

\subsection{Poincar\'e disk}\label{secpoin}
More general equations arise in the theory of harmonic mappings between Riemann surfaces~\cite{Jo, Job, Ma, Sc}. In particular, let the target be the Poincar\'{e} disk. This is the unit disk $\, \mathbb D =\{ w \in\mathbb C\colon  |w| < 1\,\}\,$ equipped with the hyperbolic metric
$\dtext s = \frac{\abs{\dtext w}}{1-\abs{w}^2}$. The associated Dirichlet integral
\[\,\mathscr E[h] = \iint _\Omega \frac{|h_z|^2 + |h_{\bar z}|^2}{(1 - |h|^2)^2}  \,\dtext x_1\, \dtext x_2
\]
is certainly infinite for homeomorphisms $\, h\,: \Omega \onto \mathbb D\,$ in the Sobolev space $\,\mathscr W^{1,2}_{\loc}(\Omega , \mathbb D)\,$. Nonetheless, it is  interesting to examine the inner-variational equation and all its  solutions, not necessarily homeomorphisms.
\[
\frac{\partial}{\partial \bar z}\, \frac{h_z\, \overline{h_{\bar z}}}{(1 - |h|^2)^2}\,  = 0\;,\;\;\;\text{for}\;\; h\in \mathscr W^{1,2}_{\textnormal{loc}}(\Omega, \mathbb D)
\]
For slightly greater generality we consider the weighted Dirichlet integral
\[
\mathscr E[h] = \iint _\Omega \big(\,|h_z|^2 + |h_{\bar z}|^2 \big ) \,\rho(z,h)  \,\dtext x_1\, \dtext x_2
\]
and its inner-variational equation
\begin{equation}\label{Var}
\frac{\partial}{\partial \bar z}\,\big [\,\rho(z, h)\,h_z\, \overline{h_{\bar z}}\,\big]\,  = \;\rho_z(z, h)\,(|h_z|^2 + |h_{\bar z}|^2 \,)\;,\;\;\;\text{for}\;\; h\in \mathscr W^{1,2}_{\textnormal{loc}}(\Omega, \mathbb D)
\end{equation}

\begin{theorem}\label{weightedHopf} Suppose $\rho = \rho(z, w) \geqslant 1$ is Lipschitz continuous in the $\,z$-variable and H\"{o}lder continuous in the $w$-variable. If  $h\in \mathscr W^{1,2}_{\textnormal{loc}}(\Omega, \mathbb D)$ is a solution of~\eqref{Var} with nonnegative Jacobian, then $h$ is locally Lipschitz continuous.
\end{theorem}
\begin{proof}
Generally speaking we are dealing with a nonhomogeneous Cauchy-Riemann equation
\begin{equation}\label{Cauchy}
\frac{\partial U}{\partial \bar z} =  u \,,\;\;\;\; \text{where}\;\;\; U = \rho(z , h )\, h_z \overline{h_{\bar z}} \,\,,\;\;\; u = \big( |h_z|^2 + |h_{\bar z}|^2 \big)\,\rho_z (z , h )
\end{equation}
The equation is elliptic, so we gain some regularity properties of $U$ and $u$. At the beginning we only know that $U,u  \in \mathscr L^1_{\textnormal{loc}}(\Omega)$. We shall recurrently improve integrability properties of these terms. First observe that $U$,  having $\frac{\partial }{\partial \bar z}\,$-derivative in $\,\mathscr L^1_{\textnormal{loc}}(\Omega)$, lies in $ \mathscr L^s_{\textnormal{loc}}(\Omega)$  for every exponent $\,1 < s < 2\,$. Then, in view of pointwise inequality $\, |h_{\bar z}|^2  \leqslant
\rho(z , h )\, |h_z|  |h_{\bar z}| = \abs{U}$, we see that $\,|h_{\bar z}|^2 \in  \mathscr L^s_{\textnormal{loc}}(\Omega)\,$. This implies that also  $\,|h_z|^2 \in  \mathscr L^s_{\textnormal{loc}}(\Omega)\,$. In this way we  gain higher integrability of the right hand side of~\eqref{Cauchy}; namely, $u = \big( |h_z|^2 + |h_{\bar z}|^2 \big)\,\rho_z (z , h ) \in \mathscr L^s_{\textnormal{loc}}(\Omega)\,$, because $\,\rho_z(z, h)\,$ is bounded. Now equation~\eqref{Cauchy} places $U$ in the space
$\, \mathscr L^{\frac{2s}{2-s}}_{\textnormal{loc}}(\Omega)\,$. This, in view of $\,|h_{\bar z}|^2 \leqslant \abs{U}$, yields $\,|h_{\bar z}|^2 \in \mathscr L^{\frac{2s}{2-s}}_{\textnormal{loc}}(\Omega)\,$. Hence $\,|h_{\bar z}|^2 \in  \mathscr L^{\frac{2s}{2-s}}_{\textnormal{loc}}(\Omega)\,$ as well. Thus we gained more integrability  of $u$; namely, $u \in
\mathscr L^{p}_{\textnormal{loc}}(\Omega)\,$, with $\, p = \frac{2s}{2-s}  > 2\,$. We again turn back to equation~\eqref{Cauchy}. This time the equation yields H\"{o}lder continuity of $U$; that is,  $\,U\in \mathscr C^\alpha_{\textnormal{loc}}(\Omega)\,$ with $\, \alpha = 1 - \frac{2}{p} > 0 \,$. Let us write the equation as
\[
h_z \,\overline{h_{\bar z}} =  \frac{\psi(z)}{\rho(z, h)}\;,\;\;\;\text{where}\;\;\;\psi \in \mathscr C^\alpha_{\textnormal{loc}}(\Omega)
\]
We observe that $\,h\,$ is also locally H\"{o}lder continuous, because $\,h_{\bar z} \in \mathscr L^{2p}_{\textnormal{loc}}(\Omega)$ with exponent $2p > 2$. The conclusion is  that the \textit{Hopf product} $\,h_z \,\overline{h_{\bar z}}\,$ is a H\"{o}lder continuous function.  By Theorem~\ref{LipHopf2},   $\, h\,$ is locally Lipschitz.
\end{proof}

\subsection{An application to nonlinear elasticity}\label{secapp}
We now turn to some fairly general energy integrals of interest in nonlinear elasticity. Given two bounded domains $\Omega$ and $\Omega^\ast$ in $\C$, we consider mappings $h \colon \Omega \to \Omega^\ast$ of Sobolev class $\W^{1,2} (\Omega)$ whose Jacobian determinant $J_h\deff \det Dh= \abs{h_z}^2- \abs{h_{\bar z}}^2$ is nonnegative. In nonlinear elasticity of isotropic materials one considers the energy of $h$ of the form
\[\mathscr E[h] = \iint_\Omega W\big(z, h, \abs{h_z}^2, \abs{h_{\bar z}}^2 \big)\]
Specifically, neo-Hookean models of elasticity~\cite{Ba0} deal with the integrands $W$ which blow up as the Jacobian determinant approaches zero.  For the sake of simplicity we forgo the dependence of $W$ on the $z$ and $h$ variables. The interested reader may generalize our considerations by including $z$ and $h$ variables to the integrand, like in \S\ref{secpoin}. For the energy integrand
\[\mathbf E (Dh)= W \big( \abs{h_z}^2, \abs{h_{\bar z}}^2\big) \]
 the inner-variational equation~\eqref{VarEq} simplifies as follows
\begin{equation}\label{starstar}
\frac{\partial}{\partial \bar z} \left[\big(W_a+W_b\big) h_z \overline{h_{\bar z}} \,  \right]+ \frac{\partial}{\partial z} \left[\abs{h_z}^2 W_a + \abs{h_{\bar z}}^2 W_b -W \right]=0
\end{equation}
where $W$ and its partial derivatives $W_a$ and $W_b$ are evaluated at $a=\abs{h_z}^2$ and $b=\abs{h_{\bar z}}^2$.
To emphasize a possible neo-Hookean character of the integrand we express it as
\[W(a,b)=\frac{F(a,b)}{(a-b)^{p-1}} \]
so
\begin{equation}\label{energyfunction}
\mathscr E[h] = \iint_\Omega \mathbf E (Dh) = \iint_\Omega \frac{F\big( \abs{h_z}^2, \abs{h_{\bar z}}^2 \big)}{ (\abs{h_z}^2 - \abs{h_{\bar z}}^2 )^{p-1} }\;,\;\;\;p \ge 1
\end{equation}
where $F=F(a,b)$ is defined and continuous in $\,\overline{\mathbb O}= \{(a,b) \colon a\ge b\ge0\}\,$ -the closure of the first octant $\,\mathbb O \deff \{(a,b) \colon a>b>0\}\,$. Let us remove the corner of $\,\overline {\mathbb O}\,$ to introduce  $\overline{\mathbb O}_\circ \deff \overline{\mathbb O} \setminus \{(0,0)\}\,$. As usual, we write
\[
\begin{split}
\abs{\nabla F}&= \abs{F_a} + \abs{F_b} \quad \mbox{and} \quad \abs{\nabla^2 F}= \abs{F_{aa}}+ \abs{F_{ab}}+ \abs{F_{bb}}
\end{split}
\]
We make the following standing assumptions on $F  \colon \overline{\mathbb O} \to [0, \infty )\,$:
\begin{align}
& F\in \mathscr C (\overline{\mathbb O}) \cap \mathscr C^2 (\overline{\mathbb O }\setminus \{(0,0)\}) \label{c1} \\
& \mbox{$F$ is homogeneous of degree } p \label{c2} \,\mbox{; that is, } \\
& F(ta,tb)=t^{p} F(a,b) \qquad \mbox{for } t \ge 0\,,  \quad a \ge b \ge 0. \nonumber
\end{align}
Furthermore, for $\,a>b>0\,$, we assume that
\begin{align}
 (a+b)^{p} \lesssim \, & F \lesssim (a+b)^{p} \label{c3}\\
 (a+b)^{p-1} \lesssim \, & F_a+F_b \le \abs{\nabla F} \lesssim (a+b)^{p-1} \label{c4} \\
& \abs{\nabla^2 F} \lesssim (a+b)^{p-2} \label{c5}
\end{align}
Here the notation $\lesssim$ refers to an inequality with the  implied constant (positive) which stays independent of $(a,b) \in \overline{\mathbb O}$.
\begin{theorem}
Let $h\in \W^{1,1}_{\loc} (\Omega)$ be an inner-stationary mapping for the energy integral~\eqref{energyfunction} with $\mathscr E[h]<\infty$, where $F$ satisfies the conditions~\eqref{c1}--\eqref{c5}.  Then $h$ is locally Lipschitz continuous. Furthermore $\mathbf E (Dh)$ is locally bounded.
\end{theorem}
\begin{proof}
First observe that
\begin{equation}
\iint_\Omega \abs{Dh}^2 \lesssim \mathscr E[h]< \infty
\end{equation}
hence $h\in \W^{1,2} (\Omega)$. Regarding inner variation, we apply formula~\eqref{starstar} to $W(a,b)= \frac{F(a,b)}{(a-b)^{p-1}}$ to see that
$\,a\, W_a + b\, W_b -W =0\,$, because of $p$-homogeneity of $\,F\,$. The equation~\eqref{starstar} takes the form
\begin{equation}\label{equ13}
\frac{\partial}{\partial \bar z} \left[\frac{F_a+F_b}{(a-b)^{p-1}}\,  h_z \overline{h_{\bar z}}  \right]=0
\end{equation}
where we note that
\[
\begin{split}
\left| \frac{F_a+F_b}{(a-b)^{p-1}}\,  h_z \overline{h_{\bar z}}   \right| &  \lesssim \frac{(a+b)^{p-1} \sqrt{a\, b}}{(a-b)^{p-1}} \lesssim \frac{(a+b)^{p}}{(a-b)^{p-1}} \lesssim \frac{F(a,b)}{(a-b)^{p-1}}\\ & \lesssim \mathbf E (Dh) \in \mathscr L^1(\Omega)
\end{split}
 \]
Thus, by Weyl's lemma,~\eqref{equ13} yields that
\begin{equation}\label{eq9}
\frac{F_a+F_b}{(a-b)^{p-1}} \, h_z \overline{h_{\bar z}}= \phi \quad \mbox{is an analytic function, actually in }\; \mathscr L^1(\Omega)\,.
\end{equation}
We also note at this point that  in view of condition~\eqref{c4}
\begin{equation}
\abs{h_{\bar z}}^2 \le \abs{h_z} \abs{h_{\bar z}} \lesssim \left| \frac{F_a+F_b}{(a+b)^{p-1}} h_z \overline{h_{\bar z}}   \right|  \le \left| \frac{F_a+F_b}{(a-b)^{p-1}} h_z \overline{h_{\bar z}}   \right| = \abs{\phi}
\end{equation}
Hence $h_{\bar z} \in  \mathscr L^\infty_{\loc} (\Omega)$. 

 We are going to solve~\eqref{eq9} for $h_{\bar z}$ in terms of $\phi$ and $h_z$, at least when $\abs{h_z}$ is sufficiently large. Choose and fix an arbitrary subdomain $\Omega' \Subset \Omega$ and let
\[N^2 \deff \|\phi \|_{\mathscr L^\infty (\Omega')}  <\infty\]
Since $F_a+F_b$ is homogeneous of degree $\,p-1\,$, equation~\eqref{eq9} can be written as
\begin{equation}\label{eq12}
\frac{F_a (1,k^2)+ F_b (1,k^2)}{(1-k^2)^{p-1}} \frac{\overline{h_{\bar z}}}{h_z} = \frac{\phi}{h_z^2}
\end{equation}
where $0 \le k= \frac{\abs{h_{\bar z}}}{\abs{h_z}}  \le \nicefrac{1}{2}$\,, provided
\begin{equation}
\abs{h_z} \ge 2N.
\end{equation}
Taking the absolute values of both sides in the equation~\eqref{eq12}, we obtain
\begin{equation}\label{eq18}
\frac{F_a(1,k^2)+F_b(1,k^2) }{(1-k^2)^{p-1}}\, k = \frac{\abs{\phi}}{\abs{h_z}^2}\deff s
\end{equation}
The left hand side represents a $\mathscr C^1$-smooth function $\Phi = \Phi (k)$, $0<k \le \nicefrac{1}{2}$\,. We see that
\[\Phi' (0) = F_a (1,0)+ F_b (1,0) \gtrsim 1\]
by the condition~\eqref{c4}. Therefore~\eqref{eq18} admits unique solution for $\,k\,$ close to 0, say
\[k=s\,  \Gamma (s),\]
provided $0 \le k \le k_\circ \leqslant \nicefrac{1}{2}$, where $\Gamma = \Gamma (s)$ is a $\mathscr C^1$-function in a small interval $\,0 \le s \le s_\circ$. Now the equation~\eqref{eq12} reads as,
\begin{equation}
\overline{h_{\bar z}} = \frac{\phi}{h_z} \Gamma \Big( \frac{\abs{\phi}}{\abs{h_z}^2}\Big), \quad \mbox{whenever } \abs{h_z} \ge \frac{N}{k_\circ}
\end{equation}
Thus we arrive  at the equation of the form ~\eqref{TheEquat}, where
\[\mathcal H (z, \xi) = \frac{\overline{\phi (z)}}{\overline{\, \xi \, }} \Gamma \Big( \frac{\abs{\phi (z)}}{ \abs{\xi}^2}\Big) \;,\quad \mbox{for } \abs{\xi} \ge R \deff \frac{N}{k_\circ},\]
All the conditions on $\mathcal H$ in Theorem~\ref{MainTheorem} are satisfied, so
\[\abs{\nabla h} \in \mathscr L^\infty_{\loc} (\Omega').\]
It remains to estimate the integrand. We have the identity
\[\frac{\mathbf E (Df)}{\abs{\phi}+ \abs{h_z}}  = \frac{F(1,k^2)}{(1-k^2)^{p-1} + \left[F_a (1,k^2) + F_b (1,k^2) \right]k}\]
regardless of whenever $\abs{h_z} \ge R\,$ holds or not, where $k= \frac{\abs{h_{\bar z}}}{\abs{h_z}} \in [0,1]$. The right hand side is bounded. Therefore,
\[\mathbf E (Dh) \lesssim \abs{\phi} + \abs{h_z} \in \mathscr L^\infty_{\loc} (\Omega) \qedhere\]
\end{proof}
\subsubsection{An example}
The class of energies in~\eqref{energyfunction} covers the following particular integral,
\[\mathscr E_\Omega [h]= \iint_\Omega \frac{\abs{Dh}^{2p}}{J_h^{\, p-1}} = \iint_{\Omega} \frac{\left( \abs{h_z}^2 + \abs{h_{\bar z}}^2 \right)^p}{\left( \abs{h_z}^2 - \abs{h_{\bar z}}^2
\right)^{p-1}}, \quad p \ge 1\]
subject to homeomorphisms $h \colon \Omega \onto \Omega^\ast$ in the Sobolev space $\W^{1,2} (\Omega)$. This case gains additional interest in Geometric Function Theory because the transition to the energy of the inverse mapping $f=h^{-1} \colon \Omega^\ast \onto \Omega$ results in the  $\mathscr L^p$-norm of the distortion function
\[\mathscr E_{\Omega^\ast } [f] = \iint_{\Omega^\ast} K_f^p\;, \qquad K_f= \frac{\abs{Df}^2}{J_f}\,\geqslant 1 \]
We see that conformal mappings, for which $\,K_f \equiv 1\,$, are the absolute minimizers. In general, $\,\mathscr L^p$-integrability of the distortion function only guarantees that $f\in \W^{1, \frac{2p}{p+1}}(\Omega^\ast)$. Indeed,
\[
\begin{split}
\iint_{\Omega^\ast} \abs{Df}^\frac{2p}{p+1} &= \iint_{\Omega^\ast} K_f^\frac{p}{p+1} \, J_f^\frac{p}{p+1} \le \left(  \iint_{\Omega^\ast} K_f^p\right)^\frac{1}{p+1} \left(\iint_{\Omega^\ast} J_f \right)^\frac{p}{p+1}\\
&=  \|K_f \|^{^\frac{p}{p+1}}_{_{\mathscr L^p (\Omega^\ast)}} \;\cdot \abs{\Omega}^\frac{p}{p+1} < \infty
\end{split}
\]
We do not pursue this matter further; see~\cite{AIM, AIMO, JM, Mar} for more on the minimization of $\|K_f\|_{\mathscr L^p}\,$.
\subsubsection{Nonisotropic energies}
The methods presented in this paper are also pertinent to some nonisotropic energies~\eqref{equ21}. By way of illustration, here is an example of such energy integral
\[\mathscr E [h]= \iint_{\Omega} \mathbf{E} (Dh) = \iint_\Omega \frac{F \big(\abs{h_z}^2, \abs{h_{\bar z}}^2, 2 \re h_z h_{\bar z}\big)}{\left(\abs{h_z}^2- \abs{h_{\bar z}}^2\right)^{p-1}}\;, \quad p \ge 1\]
where $\,F=F(a,b,c)\,$ is a function of three real variables $\,a \ge b \ge 0\,$ and $\,c\in \R\,$. We assume, as in \S\ref{secapp}, that $\,F\,$ is homogeneous of degree $\,p\,$; that is,
\[F(ta, tb, tc)=t^p F(a,b,c)\]
An elementary but tedious computation reveals that the inner-variational equation~\eqref{VarEq} takes the form :
\[\frac{(F_a+F_b)h_z \overline{h_{\bar z}}+ (h_z^2+\overline{h_{\bar z}}^2)F_c}{\left( \abs{h_z}^2-\abs{h_{\bar z}} \right)^{p-1}}=\phi \qquad \mbox{-an analytic function in } \Omega\,.\]
We leave it to the interested reader to complete this investigation by imposing precise, fairly minimal, conditions on $\,F=F(a,b,c)\,$ in order to implement Theorem~\ref{MainTheorem}. The conclusion is that $h$ is locally Lipschitz and the integrand $\mathbf E (Dh)$ is locally bounded. An explicit example is:
\[F(a,b,c)= (a+b)^p+ \epsilon\, c^2 (a+b)^{p-2}, \qquad p \ge 2\]
with $\epsilon > 0$ sufficiently small.

\section{Elaborate statement and examples}\label{elaborate}

 The nonnegativity  of Jacobian for a solution of~\eqref{TheEquat},
 under the structural conditions~\eqref{LipCond}--\eqref{HoldCond},  implies
  \begin{equation}
  \|h_{\bar z}\|_{\infty} \leqslant  M + \sqrt{L} + R
 \end{equation}
 For the sake of greater generality, and clarity of the arguments as well, we reformulate our main result
(Theorem~\ref{MainTheorem})
replacing the Jacobian condition  with the boundedness of $h_{\bar z}$.
Furthermore, we give a quantitative result with the sharp asymptotic bound  on the gradient of $h$
near the boundary. Define
\[\textnormal{\textsc{osc}}_{_\Omega} [h] \deff \sup_{a ,\, b\, \in \Omega} \big|\,h(a) - h(b)\,\big|\]
and
\[\abs{ \nabla h (z)}  \deff \abs{h_z(z)} + \abs{h_{\bar z}(z)}\]

\begin{theorem}\label{LipHopf3}
To every equation~\eqref{TheEquat} there corresponds a structural constant
$\lambda_\circ = \lambda_\circ(\mathcal{H})$ such that the following holds.
Suppose that a  function $h\in \mathscr W^{1,2}(\Omega)\cap \mathscr L^\infty(\Omega)$
with $h_{\bar z} \in \mathscr L^\infty(\Omega)$ satisfies~\eqref{TheEquat}.  Then $h$ is locally Lipschitz. Moreover, for almost every $z \in \Omega$ we have
\begin{equation}\label{mainineq}
  |\nabla h (z)|   \; \leqslant   \frac{3\,}{r}\, \textnormal{\textsc{osc}}_{_\Omega} [h]
 \;+\; 4\,\|h_{\bar z}\|_{\mathscr L^\infty (\Omega)}   \;+\; 6\,\lambda_\circ
\end{equation}
where $r= \min \{\textnormal{dist}(z , \partial \Omega) , 1\}$. In particular,
\[
\limsup_{z \rightarrow \,\partial \Omega} \;\; |\nabla h (z)|\cdot \textnormal{dist}(z , \partial \Omega)\, \,\leqslant  3\,\, \textnormal{\textsc{osc}}_{_\Omega} [h]
\]
If, in addition, $\,h\,$  is continuous up to the boundary then
\[
\limsup_{z \rightarrow \,\partial \Omega} \;\; |\nabla h (z)|\cdot \textnormal{dist}(z , \partial \Omega)\, = 0
\]
\end{theorem}

\subsubsection*{Prerequisites.}
Before proceeding to the proof of Theorem~\ref{LipHopf3} we need some definitions.

\begin{definition}\label{qrdef}
A mapping $g\in W_{\rm loc}^{1,2}(\Omega)$ is \emph{quasiregular} if there exists a constant $k<1$ such that
 $\abs{g_{\bar z}}\leqslant k\,\abs{g_{z}}$ a.e. in $\Omega$.  Such a mapping is called  \emph{quasiconformal} if it is also injective.
\end{definition}

The Jacobian determinant $J_g = \abs{g_z}^2-\abs{g_{\bar z}}^2$ of a quasiregular mapping is clearly nonnegative. It is well-known that a nonconstant quasiregular mapping is both open and discrete, see~\cite[Corollary 5.5.2]{AIMb}
or~\cite[Chapter VI]{LV}.  In the sequel we shall appeal to the following topological fact which relates the cardinality of preimages with the topological degree~\cite[Theorem 4]{He} and~\cite[Proposition~I.4.10]{Rib}.

\begin{lemma}\label{lemmadeg}
Let $\Omega \subset \C$ be a bounded domain and $\,G \colon \overline{\Omega} \to \C$ an orientation preserving continuous mapping that is open and discrete in $\Omega$. Then for every $v \not\in G(\partial \Omega)$ we have
\[\Card \{z \in \Omega \colon G(z)=v\} \leqslant \deg_{_\Omega} [v;G]\]
\end{lemma}

\subsubsection*{Examples}

Here we provide examples that demonstrate the sharpness of both assumptions and conclusions
of our theorems.

The first example shows that even the most basic of our equations,
\[
h_z \overline{h_{\bar z}}=-1
\]
admits $\,\mathscr W^{1,2}$ -solutions with $J_h \ge 0$ that are not  $\mathscr C^1$-smooth.

\begin{example}
Let $\Omega= \{z \colon \abs{z-1}<\nicefrac{1}{2}\}$. For $z\in \Omega \setminus [1,\nicefrac{3}{2})$ we define
\[h(z)= \frac{3}{2} (\bar z^{\nicefrac{2}{3}}-  z^{\nicefrac{2}{3}} ) + (1- z^{\nicefrac{2}{3}})^{\nicefrac{3}{2}}+ (1- \bar z^{\nicefrac{2}{3}})^{\nicefrac{3}{2}}    \]
using the principal branch of the power function. Then
\[
\begin{split}
h_z &= -z^{- \nicefrac{1}{3}} - z^{- \nicefrac{1}{3}} (1- z^{\nicefrac{2}{3}})^{\nicefrac{1}{2}}\\
h_{\bar z} &=  \bar z^{- \nicefrac{1}{3}} - \bar z^{- \nicefrac{1}{3}} (1- \bar z^{\nicefrac{2}{3}})^{\nicefrac{1}{2}}
\end{split}
\]
Therefore, $h_z \overline{h_{\bar z}}=-1$. Note $h$ has a Lipschitz extension to $\Omega$, in fact $h=0$ on the interval $[1,\nicefrac{3}{2})$. One can also check that $J_h\ge 0$ a.e.
However, $h \not\in \mathscr C^1(\Omega)$, because $z^{\nicefrac{1}{3}} (h_z+ \overline{h_{\bar z}}) = 2 (1-z^{\nicefrac{2}{3}})^{\nicefrac{1}{2}}$ fails to be continuous.
\end{example}

A much simpler example can be given if one does not insist on the special Hopf-Laplace structure
$\,h_z\overline{h_{\bar z}}\,$, arising from the Dirichlet energy integral.
Let
\[
h(z)=\begin{cases} 3z , &\qquad \im z\ge 0 \\
2z+\bar z, & \qquad \im z<0
\end{cases}
\]
This is a bi-Lipschitz mapping of $\C$ which is not $\mathscr C^1$-smooth but satisfies the equation
\[
h_{\bar z}=\frac{6}{h_z}-2,
\]
as well as any equation of the form $h_{\bar z}=\mathcal H(h_z)$ with $\mathcal H(2)=1$ and $\mathcal H(3)=0$.

 Theorem~\ref{LipHopf3}  also encompasses other Hopf type products such as
\begin{equation}\label{HopProd}
 \begin{split}
 & h_z  \, h_{\bar z} = \phi \;, \quad \; \;\;\phi \in \mathscr C^\alpha (\Omega)\\&
 |h_z| \,  h_{\bar z} = \psi \; , \quad \; \psi \in \mathscr C^\alpha (\Omega),
 \end{split}
 \end{equation}
 but fails for some Hopf type products, like in the \textit{pseudo Hopf-Laplace equation} below.

 \begin{example}
The equation
\be{pseudo-hoplace}
h_z \abs{h_{\bar z}} = 1
\ee
admits a solution that is quasiconformal but not locally Lipschitz.
\end{example}

\begin{proof}
We look for $h$ in the form
\be{hform}
h(z)=2 z\,\psi(-2\log\abs{z}), \qquad \abs{z} \le 1
\ee
where $\psi\colon [0,\infty)\to [1,\infty)$ is a strictly increasing
function with $\psi(0)=1$.
Since
\be{hder}
h_z = 2\, \psi(-2\log \abs{z}) - 2\, \psi'(-2\log \abs{z}) \qquad
h_{\bar z}=-\frac{2z}{\bar z} \psi'(-2\log \abs{z})
\ee
the function $\psi$ must satisfy the differential equation
\be{ode1}
(\psi(t) - \psi'(t))\psi'(t) = \frac{1}{4}, \qquad 0<t<\infty.
\ee
Rewriting~\eqref{ode1} as an equation for the inverse function
\be{ode2}
\frac{dt}{d\psi} = \frac{1}{2}(\psi+\sqrt{\psi^2-1}),\qquad t(1)=0,
\ee
we arrive at
\be{ode3}
t = \psi^2-1+ \psi\sqrt{\psi^2-1} - \log (\psi+\sqrt{\psi^2-1}).
\ee
The equation~\eqref{ode3} determines a differentiable strictly
increasing function $\psi\colon [0,\infty)\to [1,\infty)$,
since the right hand side of~\eqref{ode2} is positive. Note that
$\psi(t)\to\infty$ as $t\to\infty$, hence $h$ is not Lipschitz
in any neighborhood of the origin. In view of equation~\eqref{ode1} this implies $\psi'(t)\to 0$ as
$t\to\infty$. It now follows from~\eqref{hder} that $h$ is
quasiconformal in a neighborhood of the origin; in particular $\, J_h > 0\,$ almost everywhere.
\end{proof}
The major difference between equations (\ref{HopProd})  and  (\ref{pseudo-hoplace}) is that the latter is not solvable for $\,h_{\bar z}\,$ in terms of $\,h_z\,$.

Our final example shows that the H\"older continuity of $z \mapsto \mathcal H (z, \xi)$ in Theorem~\ref{LipHopf3} cannot be relaxed to continuity, even for the standard Hopf product.

\begin{example}
Let $h(z)= z \log \log \abs{z}^{-2}$ for $\abs{z} < \nicefrac{1}{2}$. This mapping is an orientation preserving homeomorphism which belongs to $\W^{1,p}$ for all $p< \infty$. We compute
\[
%\begin{split}
h_z  = \log \log \frac{1}{\abs{z}^2} - \log^{-1} \frac{1}{\abs{z}^2}\qquad \mbox{ and } \qquad
h_{\bar z} = \frac{z}{\bar z} \log^{-1} \frac{1}{\abs{z}^2}.
%\end{split}
\]
Clearly, $h_z \overline{h_{\bar z}}$ is continuous. However, $h$ is not Lipschitz.
 \end{example}

Even for the most basic equation $h_z \overline{h_{\bar z}} =1$,  allowing the Jacobian of $h$ to change sign
destroys any hope for improved regularity~\cite{CIKO}.

\section{Model case: the Hopf-Laplace equation}\label{model}

In order to illustrate our ideas without getting into technicalities, we first take on stage the Hopf-Laplace equation
\begin{equation}\label{Hopf2}
h_z\,\overline{h_{\bar z}} = \phi(z), \qquad \mbox{with } \phi    \text{ analytic in }  \Omega \subset \mathbb C
\end{equation}

\begin{theorem}\label{LipHopf} Suppose that the Hopf product $\,h_z\,\overline{h_{\bar z}} = \phi(z)\,$ \; is analytic and bounded in a domain $\,\Omega \subset \mathbb C\,$, for some $\, h\in \mathscr W^{1,2} (\Omega)\cap \mathscr L^\infty(\Omega)$ with $h_{\bar z} \in \mathscr L^\infty (\Omega)$. Then $\, h\,$ is locally Lipschitz. Moreover, for almost every $\, z \in \Omega\,$ we have
\begin{equation}\label{serve}
  |\nabla h (z)|   \; \leqslant \; \frac{13\,\, \textnormal{\textsc{osc}}_{_\Omega} [h] }{\textnormal{dist}(z , \partial \Omega)}
 \;+\;   \; 2\,\|h_{\bar z}\|_{\mathscr L^\infty (\Omega)}
    \;+\; 3\,\|\phi\|^{1/2}_{\mathscr L^\infty (\Omega)}
\end{equation}
\end{theorem}

\begin{proof}
 Finding good solutions to equation~\eqref{Hopf2} presents no difficulty. First consider $\,\Omega = \mathbb D\,$ -the unit disk, and assume that $\,\phi \,$ is bounded in $\,\mathbb D\,$. Denote by $\,\Phi = \Phi(z)\,$ the anti-derivative of $\,\phi\,$ such that $\,\Phi(0) = 0\,$. Thus  $\;\Phi_{\bar z} = 0\,$ and  $ \, \Phi_z  =\phi \,$. Clearly,  $\,\Phi\,$ extends continuously to the closed unit disk $\mathbf D= \overline{\mathbb D}$.  The mappings $\,F^\lambda(z) = \lambda z +  f^\lambda(z)\,$, where $\,f^\lambda(z)=  \overline{\lambda^{-1}\Phi(z)}$ with complex parameter $\, \lambda \neq 0 \,$, solve the  same equation~\eqref{Hopf2}. Also note that $\,\| f^\lambda\|_\infty  \leqslant \abs{\lambda}^{-1} \|\phi\|_\infty\,$. A short computation reveals that the difference  $\, g =  g^\lambda(z) = F^\lambda(z) - h(z) \,$ is a $\W^{1,2} (\mathbb D)$-solution to a linear Beltrami type equation
 \[
   g_{\bar z} (z ) =  \nu(z)  \,\overline{g_z(z) }\;,\;\;\; \nu(z) = \frac{ - h_{\bar z}(z)}{\overline{\lambda}}\; ,\;\;\; |\nu(z)| \leqslant \frac{1}{2}
\]
  whenever $\,|\lambda | \geqslant 2 \,\| h_{\bar z }\|_\infty \,$. Now consider a continuous family of mappings $\, G^\lambda(z) = \frac{1}{\lambda}\, g^\lambda(z) =  z + \frac{1}{\lambda} \big[ f^\lambda(z) - h(z)\big]\,$. We have
  \[
  |\,G^\lambda(z) - z \,| \leqslant \frac{\|\phi\|_\infty}{ |\lambda|^2} + \frac{\| h \|_\infty}{|\lambda|}\, < \frac{1}{3}
\]
  provided $\, |\lambda| \geqslant  2 \,\sqrt{\|\phi\|_\infty}\,$ and $\, |\lambda| \geqslant  13 \,\| h\|_\infty\,  $.  This shows, in particular, that $G^\lambda$ is a nonconstant quasiregular mapping, thus  orientation-preserving, open and discrete. At this point we appeal to a Rouch\'{e} type lemma.

\begin{lemma}\label{Rouche's Lemma}
Let $\,G = G^\lambda(z)\,$ be a continuous family of mappings $\,G^\lambda : \mathbf D \rightarrow \mathbb C\,$ parametrized by complex numbers $\,\lambda\,$ with $\,\varrho \leqslant |\lambda| \leqslant \infty\,$, such that
\begin{itemize}
\item[(i)] \quad $\, G^\infty (z) \equiv z\,$
\item[(ii)] \quad $\, |G^\lambda(z) - z |  < \frac{1}{3} \;,\;\;\text{for}\;\;\; z \in \mathbf D \;\;\text {and}\;\; |\lambda| \geqslant \varrho  \,$
\item[(iii)]  For every $\, |\lambda| \geqslant \varrho \,$ the map $\, G^\lambda  : \mathbb D \rightarrow \mathbb C\,$ is orientation preserving open and discrete.
\end{itemize}
Then, given any $\,z_\circ \in \frac{1}{3} \mathbf D\,$ and parameter $\,|\lambda|\geqslant \varrho \,$, the equation
\begin{equation}\label{Rouche}
G^\lambda(z) = G^\lambda(z_\circ) \;,\;\;\;\text{for}\;\; z \in  \mathbf D
\,;\end{equation}
admits exactly one solution $\, z = z_\circ\,$.
\end{lemma}
\begin{proof}[Proof of Lemma~\ref{Rouche's Lemma}]
For $\,|z| = 1\,$, we see from (ii) that $\,|G^\lambda(z)|  > \frac{2}{3}\,$. Fix any point $\,v\,$ of modulus $\,|v| < \frac{2}{3}\,$, so $\, v \notin G^\lambda(\partial \mathbf D) \,$. Therefore, we have well defined topological degree of $\,v$, denoted  by
\[
\textnormal{deg}_{_\mathbb D}[ v ;\; G^\lambda]\qquad  \text{for }  |\lambda| \geqslant \varrho
\]
This is an integer-valued function, continuous in $\,\lambda\,$, thus constant. For $\,\lambda = \infty\,$ the degree is 1, because $\,G^\infty(z) \equiv z\,$. Therefore
\[
\textnormal{deg}_{_\mathbb D}[ v ;\; G^\lambda]\; = 1\;,\;\;\;\text{for all parameters} \;\;|\lambda| \geqslant \varrho .
\]

Now comes another topological fact concerning orientation preserving open discrete maps. It asserts that the cardinality of the preimage of $\,v\,$ does not exceed the degree, see Lemma~\ref{lemmadeg}. In symbols,
\[
 0 \leqslant \textnormal{Card} \{ z\in \mathbb D \colon G^\lambda(z) = v\,\} \;\leqslant \textnormal{deg}_{_\mathbb D}[ v ; \; G^\lambda]\; = 1 \;,\;\;\;\text{for all} \;\; |\lambda|\geqslant \varrho .
\]
This applies to the point $\, v = G^\lambda(z_\circ)\,$, because $\,| G^\lambda(z_\circ)| \leqslant
| G^\lambda(z_\circ)\,-\,z_\circ | + |\,z_\circ\,| < \frac{1}{3} + \frac{1}{3} = \frac{2}{3}\,$, by condition (ii). The lemma is established.
\end{proof}

Returning to the proof of Theorem~\ref{LipHopf},
 we infer that the mappings $G^\lambda (z)= \frac{1}{\lambda} g^\lambda (z)$ are injective in the disk $\,\frac{1}{3}\,\mathbf D\,$. So are the mappings  $\, g^\lambda (z) = \lambda z +  \overline{\lambda^{-1}}\Phi(z) - h(z) \,$. This reads as follows:
  \begin{equation}
  h(z_1) - h(z_2) \; \neq \; \lambda \cdot\Big \{ z_1 - z_2 \;+\; \frac{1}{|\lambda|^2} \,\big[\Phi(z_1) - \Phi(z_2)  \big]\,  \Big \}
  \end{equation}
  for $\,z_1 \neq z_2\,$ in the disk $\,\frac{1}{3}\,\mathbf D\,$. Letting $\,\lambda\,$ run over a circle of radius $\,|\lambda |\,$ we conclude that
  \begin{equation}
  |h(z_1) - h(z_2)| \; \neq \; |\lambda| \cdot\Big | z_1 - z_2 \;+\; \frac{1}{|\lambda|^2} \,\big[\Phi(z_1) - \Phi(z_2)  \big]\,  \Big |
  \end{equation}
  This is possible only when
  \begin{equation}\label{eq34}
  |h(z_1) - h(z_2)| \; <  \; |\lambda| \cdot\Big | z_1 - z_2 \;+\; \frac{1}{|\lambda|^2} \,\big[\Phi(z_1) - \Phi(z_2)  \big]\,  \Big |
  \end{equation}
  because the right hand side is continuous in $\,\lambda\,$ and the inequality~\eqref{eq34}  holds for large values of  $\,|\lambda |\,$.

  A conclusion is immediate; the solution $\,h\,$ is Lipschitz continuous in  the disk $\,\frac{1}{3}\,\mathbf D\,$. Moreover,
\begin{equation}
  \|\nabla h\|_{\mathscr L^\infty (\frac{1}{3}\mathbf D)}  \; \leqslant  \; |\lambda| \;+\; \frac{1}{|\lambda|} \,\|\phi\|_ {\mathscr L^\infty (\mathbf D)}
  \end{equation}
All the conditions we have encountered for the parameter~$\,\lambda\,$ are satisfied if we set
\begin{equation}
|\lambda| =  \; \max \, \begin{cases}
2\,\|h_{\bar z}\|_{\mathscr L^\infty (\mathbf D)} \vspace{0.2cm} \\
 2\,\|\phi\|^{1/2}_{\mathscr L^\infty (\mathbf D)}\vspace{0.2cm}  \\
13\, \|h\|_{\mathscr L^\infty (\mathbf D)}
\end{cases}
\end{equation}
Therefore,
\begin{equation}\label{EstimateD}
  \|\nabla h\|_{\mathscr L^\infty (\frac{1}{3}\mathbf D)}  \; \leqslant  \; 2\,\|h_{\bar z}\|_{\mathscr L^\infty (\mathbf D)} \;+\; 13\, \|h\|_{\mathscr L^\infty (\mathbf D)} \;+ \; 3\,\|\phi\|^{1/2}_{\mathscr L^\infty (\mathbf D)}
\end{equation}
completing the analysis of the case $\Omega= \mathbb D$.

Now let $\Omega$ be a general domain. Suppose $\,\mathbf B(a,r) = \{ z \colon |z-a| \leqslant r\,\}\,$ is a closed disk contained in $ \Omega$ and $h$ is a solution to the Hopf-Laplace equation in $\, \Omega\,$. We scale down the variables to introduce a function $\,\hbar(z) = \frac{1}{r} h(rz + a)\,$ which satisfies the Hopf-Laplace equation  $\,\hbar_z\,\overline{\hbar_{\bar z}} \;=\; \phi(rz + a)\,$ in the unit disk. Inequality~\eqref{EstimateD}, applied to $\,\hbar\,$,  yields
\begin{equation}\label{EstimateB}
  \|\nabla h\|_{\mathscr L^\infty (\frac{1}{3}\mathbf B)}  \; \leqslant  \; 2\,\|h_{\bar z}\|_{\mathscr L^\infty (\mathbf B)} \;+\; \frac{13}{r}\, \|h\|_{\mathscr L^\infty (\mathbf B)} \;+ \; 3\,\|\phi\|^{1/2}_{\mathscr L^\infty (\mathbf B)}
\end{equation}
We are allowed  to subtract any constant from $\,h\,$, given that $\,h\,$ appears in~\eqref{Hopf2} only with its derivatives.
The estimate~\eqref{serve} follows.
\end{proof}

\begin{corollary}
Under the hypotheses of Theorem~\ref{LipHopf} we have
\[
\limsup_{z \rightarrow \,\partial \Omega} \;\; |\nabla h (z)|\cdot \textnormal{dist}(z , \partial \Omega)\, \,\leqslant  13\,\, \textnormal{\textsc{osc}}_{_\Omega} [h]
\]
If, in addition, $\,h\,$  is continuous up to the boundary then
\begin{equation}\label{mayfail}
\limsup_{z \rightarrow \,\partial \Omega} \;\; |\nabla h (z)|\cdot \textnormal{dist}(z , \partial \Omega)\, = 0
\end{equation}
\end{corollary}

The property~\eqref{mayfail} may fail for solutions that are not continuous up to the boundary.

\begin{example}
Here is a bounded solution  $\,h (z) =  \bar z ^{\,2}\, +\, \sin \log z   \,$ in a half-disk  $\,\Omega = \{ z\colon \re \, z > 0 \;,\; |z| < 1\,\}\,$ to the Hopf-Laplace equation
\begin{equation*}
h_z\,\overline{h_{\bar z}} \;=\; 2\,\cos \log z \;\; \quad\quad \text{-analytic and bounded }
\end{equation*}
This solution exhibits large oscillations arbitrarily close to  the point $\,a = 0 \in \partial \Omega$. It is for this  reason that $h$  fails to have a limit $ \, \lim_{z \rightarrow  0} \; (\re z)  |\nabla h (z)|\,$.
\end{example}

\section{Outline of the proof of Theorem~\ref{LipHopf3}}

The arguments presented for the proof of Theorem~\ref{LipHopf}
contain many of the key ideas of the proof of our most general result, Theorem~\ref{LipHopf3}.
The  term $\,3\,\|\phi\|^{1/2}_{\mathscr L^\infty (\Omega)}\,$ in~\eqref{serve} will be replaced by a number $\lambda_\circ=\lambda_\circ (\mathcal H) $ which depends on the equation; that is, the conditions on $\mathcal H$. Let us emphasize that $\lambda_\circ$ will not depend on the solution $h$.

However, the situation is more intricate because we need to find a counterpart of the antiderivative of $\, \phi\,$. The proof of Theorem~\ref{LipHopf}  suggests that we should
 look for the family $\, F^\lambda(z) = \lambda z + f^\lambda(z)\, $ that complies with the equation~\eqref{TheEquat}; that is,
$f_{\bar z} = \mathcal H(z, \lambda + f_z)$.  We must show that the latter equation admits a continuous family $\{f^\lambda\}$ of ``good" solutions.  The key is that $f^\lambda$ will enjoy uniform Lipschitz bounds, independent of $\,\lambda\,$. Then the proof of the Lipschitz regularity of $\, h\,$ will be carried out by topological arguments in much the same way as in the case of Hopf-Laplace equation.

We first consider the case $\Omega=\mathbb D$, and treat general domains $\Omega$ by rescaling. The sharpness of Lipschitz regularity was already demonstrated in Section~\ref{elaborate}.

\section{Good solutions}\label{goodsol}

We are looking for a family $\{f^\lambda \colon \lambda \in \C, \; \; \abs{\lambda} \geqslant \lambda_\circ \}$ of  solutions  to the equation
\begin{equation}\label{TheEquat*}
f_{\bar z} = \mathcal H(z, \lambda + f_z)
\end{equation}
in the closed unit disk $\mathbf D = \{ z \colon |z| \leqslant 1\,\}$, where   $\mathcal H = \mathcal H(z,\xi)$ is continuous in $\mathbf D \times \widehat{\mathbb C}_R$ and satisfies the structural conditions~\eqref{LipCond} and~\eqref{HoldCond} in $\Omega = \mathbb D$.

\begin{proposition}\label{GoodSol}  There is $\, \lambda_\circ = \lambda_\circ (\alpha, L, M, R)\,$  and a family $\,\{f^\lambda\}_{_{|\lambda| \geqslant \lambda_\circ}}\,$  of solutions to the equation~\eqref{TheEquat*} in $\,\mathbf D\,$ such that
\begin{align}
 f^\lambda(0) & = 0 \\
 \abs{ f^\lambda(z_1) - f^\lambda(z_2) } & \leqslant  \lambda_\circ\cdot \abs{z_1 - z_2} \label{starr} \\
 \abs{ f^{\lambda_1}(z) - f^{\lambda_2}(z) } &  \leqslant   \lambda_\circ\cdot\Big| \frac{\lambda_1 - \lambda_2}{\lambda_1\cdot\,\lambda_2 }\Big| \label{circc}
\end{align}
Furthermore, for every $\,z \in \mathbf D\,$ and $\,|\lambda| \geqslant \lambda_\circ\,$, we have
\begin{align}
\abs{\nabla f^\lambda (z)} &  \leqslant \lambda_\circ \label{eq44p}\\
|f^\lambda_{z} (z)|\, , \;  |f^\lambda_{\bar z} (z)| &\leqslant \frac{1}{2} |\lambda_\circ | \leqslant  \frac{1}{2} |\lambda| \leqslant |\lambda + f^\lambda_z(z) \,| \label{lambdaEstimate}
\end{align}
\end{proposition}
\begin{proof}
The proof of Proposition~\ref{GoodSol} is proceeded by an extension of equation~\eqref{TheEquat*} to the entire plane $\mathbb C$.

\subsection{Extension to $\mathbb C$} We set for all $z\in \mathbb C$ and $\abs{\xi} >R$,
\begin{equation}\label{ExtEquat}
\mathbf H(z, \, \xi) = \begin{cases}
\mathcal H(z, \xi) & \;\;\text{if}\;\; |z| \leqslant 1 \\
\big(2 - |z|\big) \cdot \mathcal H(1 / \bar z \,, \xi ) & \;\;\text{if}\;\; 1\leqslant |z | \leqslant 2   \\
0 & \;\;\text{for}\;\;  |z | \geqslant 2
\end{cases}
\end{equation}
It is not difficult to see that inequalities~\eqref{LipCond} and~\eqref{TheEquat} transmit to $\mathbf H$ as follows
\begin{align}
|\mathbf H(z, \xi_1) - \mathbf H(z, \xi_2) | \leqslant  L \cdot \Big|\frac{1}{\xi_1} - \frac{1}{\xi_2}\Big|\;,\;\;\text{for $\,z \in \mathbb C\,$ and $\,|\xi_1| ,\,|\xi_2| > R\,$} \label{LipCondC} \\
\underset{z\in\Omega}{\sup}\, |\mathbf H(z, \xi)| \;\;+\;\;\underset{z_1\neq z_2}{\sup} \frac{|\mathbf H(z_1, \xi) - \mathbf H(z_2, \xi)|}{|z_1 - z_2|^\alpha} \; \leqslant 6M\,,\;\;\;\text{for}\;\;|\xi| >R \label{HolCondC}
\end{align}
The verification of~\eqref{HolCondC} is a routine matter of the triangle inequality.

The desired solutions $f^\lambda \, : \mathbf D \rightarrow \mathbb C$ of~\eqref{TheEquat*} will be obtained as  restrictions to the unit disk of the solutions, still denoted by $\, f = f^\lambda \, : \mathbb C \rightarrow \mathbb C$, of the extended equation
\begin{equation}\label{TheEquatC}
f_{\bar z} = \mathbf H(z, \lambda + f_z)
\end{equation}
The advantage of passing to the extended equation lies in the use of singular integrals in the entire plane. We represent the solution in the form of the Cauchy transform of $\omega = f_{\bar z}$
\begin{equation}\label{CoTr}
 f(z) = \mathcal C \omega (z) \deff   \frac{1}{\pi} \iint_\mathbb C \Big[\frac{1}{z - \tau } \;+ \frac{1}{\tau}  \Big] \omega(\tau)
 \end{equation}
We search for the density function $\,\omega \,$ in a  Besov space $\,\mathscr B^p_\alpha(\mathbb C) \subset \mathscr L^p(\mathbb C)\,$, $\, p> \frac{2}{\alpha} > 2$. The density function will be  supported in the double disk $\, 2 \mathbf D\,$.  Recall the well  known inequality~\cite[Theorem 4.3.13]{AIMb}
\begin{equation}\label{Cauchy2}
|\mathcal C \eta (z) | \leqslant C_p \,|z|^{1 - \frac{2}{p}}\, \|\eta \|_p\;,\;\;\;z\in \mathbb C, \quad \eta \in \mathscr L^p (\mathbb C)
\end{equation}
Thus,
\begin{equation}\label{starrr}
\abs{f(z)} \leqslant C_p \|\omega \|_p , \qquad \mbox{for } z \in \mathbf D
\end{equation}

\subsection{The Besov Space $\,\mathscr B^p_\alpha(\mathbb C)\,$} Let $\,0< \alpha < 1\,$ be the H\"older exponent in~\eqref{HoldCond}.    Let us choose and fix for the rest of our paper  the integrability exponent
\begin{equation}\label{Exponent p}
 p = \frac{3}{\alpha}>3
\end{equation}
The Besov space $\,\mathscr B^p_\alpha(\mathbb C)\,$ consists of functions $\,\omega \in \mathscr L^p(\mathbb C)\,$ which satisfy
\begin{equation}
\| \,\omega \,\|_{\alpha, p} \;\deff\; \; \| \omega \|_p \; +  \;\underset{\tau \neq 0}{\sup }\, \frac{\|\, \omega(\cdot + \tau) - \omega(\cdot)   \|_p}{|\tau|^\alpha} \; <\;\infty
\end{equation}
We have a continuous imbedding $\,\mathscr B^p_\alpha(\mathbb C)\subset \mathscr L^\infty(\mathbb C) \,$ together with a uniform bound~\cite[Theorem 7.34(c)]{AF} or \cite[p. 131]{Trbook}.

\begin{equation}\label{infinityBound}
\|f_{\bar z}\|_\infty=\|\omega\|_\infty \leqslant  B_p \, \|\omega\|_{\alpha,p}
\end{equation}

\subsection{Beurling-Ahlfors Transform}
Complex derivative $\,f_z\,$ of the function in~\eqref{CoTr} is expressed by a singular integral, known as the Beurling-Ahlfors transform
\begin{equation}\label{B-ATr}
 f_z(z) = \mathcal S \omega (z) \deff  - \frac{1}{\pi} \iint_\mathbb C \frac{\omega(\tau)}{(z - \tau)^2 } \quad \text{where}\;\; \omega = f_{\bar z}
\end{equation}
We denote the norm of the operator $\,\mathcal S \,:\;\mathscr L^p(\mathbb C) \rightarrow \mathscr L^p(\mathbb C)\,$ by $\,\mathcal S_p>1$. Thus we have
\begin{equation}
\|\mathcal S \omega\|_p \leqslant \mathcal S_p\, \|\omega\|_p \;, \qquad \text{hence} \quad \|\,\mathcal S\omega \,\|_{\alpha, p} \leqslant \mathcal S_p \|\,\omega \,\|_{\alpha, p}
\end{equation}
This combined with~\eqref{infinityBound}  yields

\begin{equation}\label{infSbound}
\|f_z\|_\infty=\|\mathcal S \omega\|_\infty\; \leqslant    \mathcal S_p \, B_p \,\|\omega\|_{\alpha,p}
\end{equation}
\subsection{The structural parameter $\lambda_\circ$} We are now ready to reveal  the conditions on  the  complex parameter $\lambda$; namely,
\begin{equation}
|\lambda| \geqslant \lambda_\circ \;\deff \; \max \, \left\{\begin{array}{ll}
 \sqrt{16 \,\mathcal S_p}  \geqslant 4  & \;\;\text{condition}\;(\Lambda_1) \\
 \sqrt{81\,\mathcal S_p \,L} \geqslant  \sqrt{ 8\, \mathcal S_p \, L} & \;\;\text{condition}\;(\Lambda_2)\\
120\, \mathcal S_p\, B_p\,  M & \;\;\text{condition}\;(\Lambda_3)\\
 3 \,R & \;\;\text{condition}\;(\Lambda_4)\\
 32\, C_p \,L  \geqslant 32 L \, & \;\;\text{condition}\;(\Lambda_5)
\end{array} \right.
\end{equation}
These are not the optimal numerical values, but they are chosen for the clarity in subsequent computations.

\subsection{Solving the Extended Equation} The equation~\eqref{TheEquatC} is now equivalent to the integral equation
\begin{equation}\label{TheEquatCint}
\omega = \mathbf H(z, \lambda + \mathcal S \omega)
\end{equation}
for a density function $\,\omega = \omega^\lambda(z)\,$, which we shall find uniquely in the set
\begin{equation}\label{Ball}
\,\mathfrak B = \{ \omega \colon \|\omega\|_{\alpha, p} \,\leqslant 60\cdot M\,\}\; \subset \mathscr B_\alpha^p(\mathbb C) \subset \mathscr L^p(\mathbb C) \,
\end{equation}
This is a closed  subset of $\,\mathscr L^p(\mathbb C)\,$. We shall view $\mathfrak B$ as a complete metric space with respect to $\,\mathscr L^p$ -norm.  First observe that for each $\,\omega \in \mathfrak B\,$ we have
\begin{equation}\label{Somega}
\,\|\mathcal S \omega \|_\infty \leqslant 60\, \mathcal S_p \,B_p\, M\;\leqslant \frac{1}{2} \lambda_\circ\;,
\end{equation}
by~\eqref{infSbound} and condition~($\Lambda_3$).  This combined with conditions $\,(\Lambda _3)\,$ and $\,(\Lambda _4)\,$ yields
\begin{equation}\label{lambdaEstimate2}
 \abs{ \lambda + \mathcal S \omega (z)}
\geqslant \;\frac{1}{2} |\lambda|\; \geqslant \;\frac{1}{2} \lambda_\circ  >  R
\end{equation}
In particular, the equation~\eqref{TheEquatCint} is well defined on $\,\mathfrak B\,$.
We now introduce a nonlinear  operator $\,\texttt{T} \,:\,\,\mathfrak B \rightarrow \mathscr L^p(\mathbb C)\,$ by the rule
\begin{equation}\label{TheEquatCm}
 \texttt{T} \omega = \mathbf H(z, \lambda + \mathcal S \omega)
\end{equation}
Clearly, $\, \textnormal{supp}\, \texttt{T}\omega \subset 2\mathbf D\,$. Note, by condition ~\eqref{HolCondC},  that
\begin{equation}\label{TheEquatCmm}
\|\texttt{T} \omega\|_p \,= \,\|\mathbf H(z, \lambda + \mathcal S \omega)\|_p\, \leqslant 6 M\cdot \|\chi_{_{2\mathbf D}}\|_p = 6\,(4\pi)^{\alpha/3} M\leqslant 18\cdot M
\end{equation}
Next we estimate the difference quotient for $\,\texttt{T} \omega\,$ . To this effect we consider two cases

\textit{Case 1} : $\,|\tau| \leqslant 1\,$. For all $z\in \mathbb C$, we can write
\begin{equation}\label{eq678}
\begin{split}
& |(\texttt{T} \omega)(z+\tau) - (\texttt{T} \omega)(z)\,| \\& \leqslant  |\mathbf H\big(z+\tau , \;\lambda + \mathcal S \omega (z+\tau)\big) \; -\; \mathbf H\big(z+\tau ,\; \lambda + \mathcal S \omega (z)\big)\,| \\ & +  |\mathbf H\big(z+\tau , \lambda + \mathcal S \omega (z)\big) \; -\; \mathbf H\big(z , \lambda + \mathcal S \omega (z)\big)\,|
\end{split}
\end{equation}
The first summand will be  estimated by using~\eqref{LipCondC} and~\eqref{lambdaEstimate2},
\[   |\mathbf H\big(z+\tau , \;\lambda + \mathcal S \omega (z+\tau)\big) \; -\; \mathbf H\big(z+\tau ,\; \lambda + \mathcal S \omega (z)\big)\,|  \leqslant \frac{4 L}{|\lambda|^2}\, \big| \mathcal S \omega (z+\tau) \; -\;  \mathcal S \omega (z)\,\big|\]
For the second summand in~\eqref{eq678} we use H\"older's estimate in~\eqref{HolCondC}
\[ |\mathbf H\big(z+\tau , \lambda + \mathcal S \omega (z)\big) \; -\; \mathbf H\big(z , \lambda + \mathcal S \omega (z)\big)\,|  \leqslant   6 M |\tau |^\alpha \]
We add both summands to obtain
\[
|(\texttt{T} \omega)(z+\tau) -  (\texttt{T} \omega)(z)\,|  \leqslant \frac{4 L}{|\lambda|^2}\, \big| \mathcal S \omega (z+\tau) \; -\;  \mathcal S \omega (z)\,\big|\;+\; 6 M |\tau |^\alpha \cdot \chi_{_{3 \mathbf D}}(z)
\]
The introduction of the factor $\,\chi_{_{3 \mathbf D}}(z)\,$ is legitimate beacuse the left hand side vanishes  outside the disk $\,3 \mathbf D \,$. We now compare the $\,\mathscr L^p(\mathbb C)\,$ -norms (with respect to $z$-variable) of both sides.
\[
\begin{split}
\|\;(\texttt{T} \omega)(\cdot +\tau) - & (\texttt{T} \omega)(\cdot)\,\,\|_p \\ \leqslant \;
&\frac{4 L \mathcal S_p }{|\lambda|^2}\, \|\,\omega (\cdot+\tau) \; -\;  \omega (\cdot)\,\|_p\;+\; 6 M |\tau |^\alpha \cdot (9 \pi) ^{1/p}
\end{split}
\]
Hence, by the condition $(\Lambda_2)$,
\[
\begin{split}
\frac{\|\,(\texttt{T} \omega)(\cdot+\tau) - (\texttt{T} \omega)(\cdot)\,\|_p} {|\tau|^\alpha} &\leqslant
\frac{4 L \mathcal S_p }{|\lambda|^2}\cdot 60\, M\;+\; 33 \,M  \leqslant 36 M
\end{split}
\]
which together with~\eqref{TheEquatCmm} yields
\[
\|\,\texttt{T} \omega\,\|_{ p }\; + \frac{\|\,(\texttt{T} \omega)(\cdot+\tau) - (\texttt{T} \omega)(\cdot)\,\|_p} {|\tau|^\alpha}\; \leqslant 54\, M \;\leqslant 60\, M
\]

\textit{Case 2} : $\,|\tau| \geqslant 1\,$.  It suffices to use rough $\mathscr L^p$-bound

\begin{equation}\nonumber
\|\,\texttt{T} \omega\,\|_{ p }\; + \frac{\|\,(\texttt{T} \omega)(\cdot+\tau) - (\texttt{T} \omega)(\cdot)\,\|_p} {|\tau|^\alpha} \,\leqslant  3 \,\| \texttt{T} \omega \|_p  \leqslant  54\, M \;\leqslant 60\, M
\end{equation}
Now we see that in both cases 1 and 2  we have a desired estimate $\|\,\texttt{T} \omega\,\|_{\alpha, p } \;\leqslant \;60 \,M$, meaning that
\begin{equation}
\texttt{T} \colon \mathfrak B \rightarrow  \mathfrak B
\end{equation}

Next we show that $\,\texttt{T} \,:\,\mathfrak B \rightarrow  \mathfrak B\,$ is a contraction in the $\,\mathscr L^p$-norm. Let $\,\omega_1\,, \,\omega_2\,\in \mathfrak B\,$.  By~\eqref{LipCondC} and~\eqref{lambdaEstimate2}  we see that

\begin{equation}
\begin{split}
\|\,\texttt{T}\omega_1 \;-  \texttt{T}\omega_2\,\|_p \; &\leqslant  \;\;L\;\Big \|\frac{1}{\lambda + \mathcal S \omega_1} -  \frac{1}{\lambda + \mathcal S \omega_2}  \Big \|_p \;\\ & \leqslant \frac{4 L}{|\lambda|^2}\; \big \| \mathcal S (\omega_1 - \omega_ 2) \big \|_p \\& \leqslant \frac{4 L\, \mathcal S_p }{|\lambda|^2}\; \big \| \,\omega_1 - \omega_ 2\, \big \|_p
\end{split}
\end{equation}
Thus the contraction constant is at most $\,\frac{1}{2}\,$,  because of the condition $\,(\Lambda_ 2)$.

By virtue of  Banach contraction principle  the equation ~\eqref {TheEquatCint}  has exactly one solution $\, \omega = \omega^\lambda \in \mathfrak B$.  We define $ f^\lambda (z) \deff \mathcal C \omega^\lambda (z)$.

It remains to examine the properties of $f^\lambda$.  Using~\eqref{infinityBound},~\eqref{infSbound} and~\eqref{Ball} we estimate its derivatives as follows.
\[
\| f^\lambda_{\bar z} \|_\infty \le B_p \, \|\omega^\lambda \|_{\alpha, p} \leqslant 60 B_p \, M < \frac{1}{2} \lambda_\circ \quad \mbox{ for all } z \in \C
\]
and
\[
\| f^\lambda_{z} \|_\infty  = \|\mathcal S f^\lambda_{\bar z} \|_\infty \leqslant \mathcal S_p \, B_p \, \|\omega^\lambda \|_{\alpha, p} \leqslant 60 \mathcal S_p \, B_p \, M \leqslant \frac{1}{2} \lambda_\circ, \quad z \in \C
\]
Thus
\begin{equation}
\|\,\nabla f^\lambda \,\|_\infty \;\leqslant \; \| f^\lambda_z\|_\infty +  \| f^\lambda_{\bar z}\|_\infty  <  \lambda_\circ
\end{equation}
which implies
\begin{equation}\label{ssssstar}
\abs{f^\lambda (z)} < \lambda_\circ , \qquad z \in \mathbf D
\end{equation}
This proves~\eqref{starr} in Proposition~\ref{GoodSol}. Concerning continuity with respect to the parameter $\, \lambda\,$, let $\, |\lambda_1| ,\; |\lambda_2| \,\geqslant \lambda_\circ$.
\begin{equation*}
\begin{split}
\|\,\omega^{\lambda_1} \;- \;\omega^{\lambda_2}\,\|_p \; &\leqslant  \;\;L\;\Big \|\frac{1}{\lambda_1 + \mathcal S \omega^{\lambda_1}} -  \frac{1}{\lambda_2 + \mathcal S \omega^{\lambda_2}}  \Big \|_{\mathscr L^p(2\mathbf D)} \;\\ & \leqslant \frac{4 L}{|\lambda_1 \lambda_2|}\; \big \| (\lambda_1 - \lambda_2 ) \;+\;\mathcal S (\omega^{\lambda_1} - \omega^{\lambda_ 2}) \big \|_{\mathscr L^p(2\mathbf D)} \\&  \leqslant \,\frac{4 L}{|\lambda_1 \lambda_2|}\; \big \| \,\lambda_1 - \lambda_2 \, \|_{\mathscr L^p(2\mathbf D)}\;+\; \frac{4 L\, \mathcal S_p }{|\lambda_\circ |^2 }\; \big \| \,\omega^{\lambda_1} - \omega^{\lambda_ 2}\, \big \|_{\mathscr L^p (\mathbb C)} \\&
\,\leqslant \,\frac{16 L}{|\lambda_1 \lambda_2|}\; | \lambda_1 - \lambda_2  | \;+\; \frac{1}{2}\; \big \| \,\omega^{\lambda_1} - \omega^{\lambda_ 2}\, \big \|_p
\end{split}
\end{equation*}
Hence
\[
\|\,\omega^{\lambda_1} \;- \;\omega^{\lambda_2}\,\|_p \; \leqslant \; \frac{32\, L}{|\lambda_1 \lambda_2|}\; | \lambda_1 - \lambda_2 |
\]
Finally, we recall the estimate~\eqref{Cauchy2} for Cauchy transform, $\abs{\mathcal C \eta (z)}  \leqslant C_p \,|z| ^{1 - \frac{2}{p}}\, \| \eta \|_p$ for all $z \in \mathbb C$, which we apply to $\eta = \omega^{\lambda_1}- \omega^{\lambda_2}$. This yields

\[
|\,f^{\lambda_1}(z) \;- \;f^{\lambda_2}(z)\,| \; \leqslant \; \frac{32\,C_p L}{|\lambda_1 \lambda_2|}\; | \lambda_1 - \lambda_2 |  \;\leqslant \lambda_\circ \Big|\frac{\lambda_1 - \lambda_2}{\lambda_1\, \lambda_2} \Big| \;,\;\;\;\text{for all}\;\;\; z \in \mathbf D
\]
by the condition $(\Lambda_5)$,
establishing the inequality~\eqref{circc}  in Proposition~\ref{GoodSol}.
\end{proof}

We remark that $f^\lambda$ is $\mathscr C^{1, \alpha}$-smooth.

\section{The difference of two solutions}\label{difference}
In this section we consider the difference of two solutions of~\eqref{TheEquat}, namely the given solution $h$ and the good one $F^\lambda (z) = \lambda z+f^\lambda (z)e$ constructed in the previous section. Precisely, let
\begin{equation}\label{maybeeven}
g^\lambda (z) = \lambda z + f^\lambda (z)-h(z)
\end{equation}
We shall estimate the distortion of $\,g^\lambda\,$ under the assumption that
\begin{equation}\label{Abound}
\| h_{\bar z }(z) \|_{\mathscr L^\infty (\mathbf D)} \deff N< \infty
\end{equation}
To this effect we must impose additional  bound from below on the complex parameters $\,\lambda\,$, which will now depend on the solution $\,h\,$.

\begin{proposition}\label{distpro} Let $h$ be as in Theorem~\ref{MainTheorem} and  $f^\lambda$ as in Proposition~\ref{GoodSol}. If
\begin{equation}\label{LambdaBound}
|\lambda| \geqslant  4 N + 4 \lambda_\circ +R
\end{equation}
then the difference function $\,g^\lambda$\, in ~\eqref{maybeeven} satisfies
\begin{equation}\label{BeltramiInequality}
\big|\, g^\lambda_{\bar z} (z) \,\big|\; \leqslant \;\frac{1}{2}\; \big|\, g^\lambda_z (z) \,\big|\;, \;\;\;\; \text{almost everywhere in $\,\mathbb D\,$ }
\end{equation}
\end{proposition}
\begin{proof}
Fix any point $\,z\in \mathbb D\,$ at which $h_z$ and $h_{\bar z}$ are defined.
To simplify writing we omit the superscript $\,\lambda\,$, so from now on  $\, f = f^\lambda\,$ and $\,  g = g^\lambda \,$. Thus
 \[
 g_{\bar z} = f_{\bar z} \;- \;h_{\bar z}\;\quad\quad \text{and}\quad\quad  g_z = \;\lambda + f_z \, - \,h_z
 \]
 We begin with a simple case

 \textit{Case 1}:  $\abs{h_z} \leqslant R$. The computation goes as follows:
 \[
 \begin{split}
 2 \abs{g_{\bar z}} &\leqslant 2 \left( \abs{f_{\bar z}} + \abs{h_{\bar z}}\right) \leqslant 2 \left( \frac{1}{2} \lambda_\circ +N \right)\\
 & \leqslant \abs{\lambda} -  \frac{1}{2} \lambda_\circ -R \leqslant \abs{\lambda} - \abs{f_z}- \abs{h_z} \leqslant \abs{g_z}
 \end{split}
 \]

 \textit{Case 2}: $\abs{h_z} > R$, so we may apply the equation~\eqref{TheEquat}. First, we have the identity
 \[
 (\lambda + f_z) \,g_{\bar z} =  ( g_z  + h_z) \,g_{\bar z} = g_z \,g_{\bar z}  + \;h_z\,(f_{\bar z} - h_{\bar z} )
 \]
 Hence, by~\eqref{lambdaEstimate}
 \begin{equation}\label{Est}
 \frac{1}{2} |\lambda| \cdot|g_{\bar z}| \leqslant |\lambda + f_z|\cdot | g_{\bar z} | \leqslant  | g_z |\cdot | g_{\bar z} |  + \; | h_z |\cdot |f_{\bar z} - h_{\bar z} |
 \end{equation}
 Here in the right hand side we estimate the term $\,| g_{\bar z} |\,$ by using~\eqref{Abound} and~\eqref{eq44p}
 \[
 | g_{\bar z} | \leqslant | h_{\bar z} | \;+\;| f_{\bar z} | \leqslant N + \frac{1}{2} \lambda_\circ
 \]
 For the second term in the right hand side of~\eqref{Est} we appeal to the equation~\eqref{TheEquat} and the Lipschitz condition~\eqref{LipCond},
 \begin{equation}
 \begin{split}| h_z |\cdot |f_{\bar z} - h_{\bar z} |& =  | h_z |\cdot \big| \mathcal H(z ,\lambda + f_z)  - \mathcal H(z , h_z) \big | \\&       \leqslant \frac{L |\lambda + f_z - h_z |}{|\lambda + f_z |} = \frac{L |g_z |}{|\lambda + f_z |} \leqslant  \frac{2 L}{\lambda_\circ } |g_z |\, \leqslant  \,\frac{1}{2}\,\lambda_\circ  \,|g_z|
 \end{split}
 \end{equation}
 For the last two estimates we have used ~\eqref{lambdaEstimate}  and condition ($\Lambda_2$).
 Returning to~\eqref{Est}, in view of $\,|\lambda + f_z| \geqslant \frac{1}{2} |\lambda|\,$,   we arrive at the desired estimate
\[
 |\lambda|\cdot | g_{\bar z} | \leqslant  ( 2N + 2 \lambda_\circ ) \cdot |g_z|\;\;\;
 \]
 That is,
 \[
  | g_{\bar z} | \leqslant  \frac{ 2N +  2\lambda_\circ }{|\lambda|} \cdot |g_z| \;\leqslant \frac{1}{2} \; | g_z | \qedhere
 \]
\end{proof}

\begin{corollary}\label{ogu}
Under the assumptions of Proposition~\ref{distpro} the mapping $g^\lambda$ is quasiregular, hence either constant or both open and discrete.
\end{corollary}

\subsection{Topological degree}

Consider a continuous family of mappings
\begin{equation}\nonumber
\begin{cases}
G^\lambda(z) = \frac{1}{\lambda} \; g^\lambda (z) =  z + \frac{1}{\lambda}\;\big [ f^\lambda(z) \;-\; h(z)  \big] \,,\;\text{for}\;\; z \in \mathbf D \;\;\text{and}\;\; |\lambda| \geqslant \sigma \\
G^\infty(z)  \equiv z
\end{cases}
\end{equation}
where
\begin{equation}\label{sigma}
\sigma = 3\| h \|_{\mathscr L^\infty(\mathbf D)}  + 4 \|h_{\bar z} \|_{\mathscr L^\infty(\mathbf D)}
+5\lambda_\circ
\end{equation}

From~\eqref{ssssstar} we have the uniform bound
\[
 |G^\lambda(z)  - z  |  =  \frac{1}{|\lambda|}\cdot\big | f^\lambda(z) \;-\; h(z)  \big | \leqslant  \frac{1}{|\lambda|} \cdot ( \lambda_\circ + \|h\|_\infty )   < \frac{1}{3}
\]
By Corollary~\ref{ogu} the mapping $ \, G^\lambda  \colon \mathbf D \rightarrow \mathbb C\,$ is quasiregular. Furthermore,
it is nonconstant because for $\,\frac{1}{3} < |z| < 1\,$ we have $\,|G^\lambda(z) | > |z| - \frac{1}{3} > 0 = |G^\lambda(0)|$.
Thus $\,G^\lambda(z)\,$ is open and discrete.  By Lemma~\ref{Rouche's Lemma} we conclude that  $\,G^\lambda(z_1) \neq G^\lambda(z_2)\,$, whenever $\,z_1\,$  and $\,z_2\,$ are distinct points in $\,\frac{1}{3} \mathbf D\,$ and
$\,|\lambda| \ge \sigma \,$.
This reads as follows

\begin{corollary} For all complex parameters $\,\lambda\,$ with $\,|\lambda| \ge\sigma $ the mappings $\,g^\lambda (z) = \lambda z +  f^\lambda(z) - h(z)\,$ are injective in the disk $\,\frac{1}{3}\mathbf D\,$; that is,
for $z_1\neq z_2\;$ in $\frac{1}{3}\mathbf D$
\begin{equation}\label{inequalities}
h(z_1) - h(z_2)\; \neq \; \lambda(z_1 - z_2)  + f^\lambda(z_1) - f ^ \lambda(z_2)\
\end{equation}
\end{corollary}
We shall infer from this, using topological degree arguments,  the following inequality

\begin{lemma}\label{TheInequalitylem} For every circle $\mathbb T_\rho= \{\lambda \colon \abs{\lambda}=\rho\}$  with $\,\rho \geqslant \sigma\,$ there exists $\,\lambda\in \mathbb T_\rho\,$ such that
\begin{equation}\label{TheInequality}
|h(z_1) - h(z_2)|\; \leqslant\; |\lambda(z_1 - z_2)  + f^\lambda(z_1) - f ^ \lambda(z_2) |
\end{equation}
\end{lemma}

\begin{proof} This inequality certainly holds for large values of $\,\rho\,$. To simplify writing we denote $\,a =  h(z_1) - h(z_2)\,$ and assume, as we may, that $\,a \neq 0\,$.
We shall consider  a family of mappings $ \,\Phi_\rho^a \,: \mathbb T \rightarrow \mathbb T\,$, with parameter
$\,\rho \geqslant \sigma\,$, given by
\begin{equation}\nonumber\label{winding}
\Phi_\rho^a(e^{i \theta}) = \frac{ F(\rho \,e^{i \theta}) - a}{ |F(\rho\, e^{i \theta}) - a |} \,,\;\;\text{where} \;\;
F(\lambda) = \lambda \cdot (z_1 - z_2)  + f^{\lambda}(z_1) - f ^ {\lambda}(z_2)
\end{equation}
By virtue of the inequalities~\eqref{inequalities}, each such mapping has well defined degree, denoted by $\,\textnormal{deg}\,\Phi_\rho^a\,$, also known as winding number.  Letting the parameter $\,\rho\,$ vary  we obtain an integer-valued continuous function in $\,\rho\,$, thus constant. We identify this constant by letting $\,\rho\rightarrow \infty\,$. The mappings converge uniformly  to $ \,\Phi_\infty ^a \colon \mathbb T \rightarrow \mathbb T\,$, where $ \,\Phi_\infty ^a (e^{i \theta}) \deff \frac{z_1 - z_2}{| z_1 - z_2 |}\cdot e^{i \theta}\,$ . The degree of this limit map is equal to 1. Hence  we conclude that
\[
 \textnormal{deg}\,\Phi_\rho^a \; = 1 \,,\;\;\;\text{for all parameters} \;\;
 \,\rho \geqslant \sigma
\]
We now fix $\,\rho \geqslant \sigma$ and move the point $\,a \neq 0\,$  to $\,\infty\,$ along the straight half-line  $\,\{ t a \colon  t \geqslant 1\,\}\,$, to observe  that for some $\,t\geqslant 1\,$ the point $\,t a\,$ lies in $\,F(\mathbb T_\rho )\,$. For if not,
 we would have well defined degree of the mappings $\,\Phi_\rho^{ta} :  \mathbb T \rightarrow \mathbb T\,$, given by
\begin{equation}\nonumber
\Phi_\rho^{ta}(e^{i \theta}) = \frac{ F(\rho e^{i \theta}) - t a}{ |F(\rho e^{i \theta}) - t a |}
\end{equation}
By virtue of continuity with respect to the parameter $\,t\,$ we would have
\[
\textnormal{deg}\,\Phi_\rho^{ta} \; = \;  \textnormal{deg} \,\Phi_\rho^a \; = 1\,, \;\;\text{for all}\; t \geqslant 1
\]

On the other hand letting $\,t \rightarrow \infty\,$ the mappings $\,\Phi_\rho^{ta} \,: \mathbb T \rightarrow \mathbb T\,$ converge uniformly to a constant map $\,\Phi_\rho^{\infty} = \frac{a}{|a|}\,$, whose degree is zero, in contradiction with the case $\, t = 1\,$. Thus $ \,ta \in F(\mathbb T_\rho)\,$, for some $\,t\geqslant 1\,$, meaning that
\begin{equation}\label{somename}
t a =  \lambda \cdot (z_1 - z_2)  + f^\lambda(z_1) - f ^\lambda(z_2)\;,\;\;\;\text{for some}\;\;\,\lambda \in \mathbb T_\rho
\end{equation}
which yields the desired inequality~\eqref{TheInequality}.
\end{proof}

\section{Proof of Theorem~\ref{LipHopf3}: the final step}\label{finally}

We invoke Lemma~\ref{TheInequalitylem} with $\rho=\sigma$, where $\sigma$ is given by~\eqref{sigma}.
Thus
\[
|h(z_1) -h(z_2)| \leqslant t |a| \leqslant   \sigma\,| z_1 - z_ 2| + | f^{\,\rho \,e^{i \theta}}(z_1) - f^{\,\rho\, e^{i \theta}} (z_2) |
\leqslant (\sigma+\lambda_\circ) \abs{z_1 - z_2}
\]
where the latter inequality follows from Proposition~\ref{GoodSol}. In terms of the gradient of $\,h\,$ this reads as,
 \begin{equation}\label{gradEst}
 |\nabla h(z)| \leqslant \sigma+\lambda_\circ
 = 3\, \| h\|_{\mathscr L^\infty(\mathbf D)}  + \;4\, \|h_{\bar z}\|_{\mathscr L^\infty(\mathbf D)}  +  6\, \lambda_\circ
\end{equation}
  The final procedure, like for the  Hopf-Laplace equation, consists of rescaling the variables.

  Let $\,\mathbf B(a,r) = \{ z \colon |z-a| \leqslant r \leqslant 1\,\}\,$ be a closed disk contained in $\, \Omega\,$ and $\, h\,$ a solution to the equation $ \,h_{\bar z} = \mathcal H(z, h_z)\,$ in $\, \Omega\,$. Consider the  function $\,\hbar(z) = r^{-1} h(rz + a)\,$. It satisfies an equation  $ \,\hbar_{\bar z} = \widetilde{\mathcal H}(z,\hbar_z)\,$ in the unit disk, where $\,\widetilde{\mathcal H}(z,\xi) = \mathcal H (r z + a, \xi )\,$. We stress that $\, 0 < r \leqslant 1\,$,  so all the conditions imposed on  $\mathcal H(z,\xi)$, including~\eqref{HolCondC}, transmit to  $\, \widetilde{\mathcal H}(z,\xi)\,$,  with the same  constants; easily verified.  In other words the structural parameter $\,\lambda_\circ\,$ remains the same. Now inequality~\eqref{gradEst} applied to $\,\hbar\,$  yields
\begin{equation}
  \|\nabla h\|_{\mathscr L^\infty (\frac{1}{3}\mathbf B)}  \; \leqslant  \;  \frac{3}{r}\, \|h\|_{\mathscr L^\infty (\mathbf B)} \;+ \; 4\,\|h_{\bar z}\|_{\mathscr L^\infty (\mathbf B)} \;+\; 6\,\lambda_\circ
\end{equation}
One may subtract a suitable constant from $h$, arriving at~\eqref{mainineq}. \qed

\bibliographystyle{amsplain}

\end{document}